\documentclass[preprint,11pt]{amsart}
 \usepackage{amssymb}
\usepackage{amsmath,amsfonts,amscd}
\usepackage{mathrsfs,hyperref, enumerate}

\usepackage[a4paper,bindingoffset=0.2in,left=1in,right=1in,top=1in,bottom=1in,footskip=.25in]{geometry}

\usepackage{amsmath,amsfonts,amsthm,amssymb,amscd}
\usepackage{bbm}
\usepackage{mathrsfs}
\usepackage{color}
\usepackage{hyperref}
\usepackage{tikz-cd}
\usepackage{braket}

\usetikzlibrary{matrix}
\usepackage{graphicx}
\usepackage{enumerate}
\usepackage{changepage}

\def\C{\mathbb{C}}
\def\F{\mathbb{F}}
\def\P{\mathbb{P}}

\newcommand{\CalC}{\mathcal{C}}
\newcommand{\CalF}{\mathcal{F}}

\theoremstyle{definition}
\newtheorem{thm}{Theorem}
\numberwithin{thm}{section}

\newtheorem{prop}[thm]{Proposition}
\newtheorem{lem}[thm]{Lemma}
\newtheorem{remark}{Remark}

\newcommand{\gam}{\gamma}

\newcommand{\End}{\operatorname{End}}

\newcommand{\abs}[1]{\left| #1 \right|}

\newcommand{\tit}[1]{\textit{#1}}

\newcommand{\trm}[1]{\textrm{#1}}
\newcommand{\al}{\alpha}

\def\F{\mathbb{F}}

\newcommand{\lam}{\lambda}

\newcommand{\tr}{\textrm{Tr}}

\newcommand{\sub}{\subset}

\newcommand{\Z}{\mathbb{Z}}
\begin{document}

\begin{abstract}
Let $n,k$ be  positive integers and $q$ be a prime power.  We prove that there exists a real constant $C_k$ depending only on $k$ such that if  $ \# M_n(\F_q) > C_k$ then every element in $M_n(\F_q)$ is a sum of \tit{two} $k$th powers, proving Larsen's conjecture.
\end{abstract}

\title{Matrix Waring Problem-III}

\author{Krishna Kishore and Sunil Kumar Mallick}

\email{kishorekrishna@iittp.ac.in}
\address{Krishna Kishore,
Department of Mathematics and Statistics,
Indian Institute of Technology Tirupati,
Tirupati 517619, Andhra Pradesh,
India.}

\email{mallicksunil200114@gmail.com}
\address{Sunil Kumar Mallick,
Department of Mathematics and Statistics,
Indian Institute of Technology Tirupati,
Tirupati 517619, Andhra Pradesh,
India.}
\maketitle

\section{Introduction}

The classical Waring problem states that every natural number is a sum of $4$ squares, of $9$ cubes, of $19$ biquadrates, etc.  In other words, for all integers $k \geq 1$, if $g(k)$ is smallest positive integer such that every natural number is a sum of at most $g(k)$ $k$th powers (of natural numbers), then, in particular, $g(2) = 4$,  $g(3) = 9$, $g(4) = 19$. The existence of $g(k)$ for arbitrary $k$ was proved by Hilbert, and later on explicit formulae for $g(k)$ have been found for all but finitely many values of $k$. In this context, one may ask whether there exists a smallest bound $G(k)$ such that every \tit{sufficiently large} natural number is a sum of at most $G(k)$ $k$th powers. Apart from the trivial case $k = 1$, the exact values of $G(k)$ are known only for $k = 2$ and $k = 4$; for a survey of results about Waring's problem and its variants see \cite{VW}.

Modern versions consider Waring's problem over non-commutative structures. For example, Shalev \cite{Sh} proved that, for every fixed nontrivial group word $w$, every element of a sufficiently large finite nonabelian simple group is a product of three values of $w$. Larsen, Shalev and Tiep \cite{LST} subsequently reduced the number of factors to two.  Waring’s problem makes sense over arbitrary rings, and in particular for matrix rings
over finite fields. Let $k$ and $n$ be positive integers and let $G'(k,n)$ be the smallest positive integer such that for all \tit{suﬃciently large $q$} every matrix in $M_n(\F_q)$ can be written as a sum of $G'(k,n)$ kth powers (of matrices in $M_n(\F_q)$). It is well-known, due
to Small \cite{Sm}, that $G'(k, 1) = 2$ for all $k \geq 2$. In \cite{Ki}, the first author showed that for all $k \geq 2$, $G'(k,1) = G'(k,2) =2$ and $G'(k,n) \leq 3$ for all $n \geq 3$.
In \cite{KS}, the first author and Singh improved this result further, and proved the following \tit{weaker} version of Larsen's conjecture:
\begin{quote}
Let $n,k$ be positive integers and $q$ be a prime power. There exists a real constant $C_k$ depending only on $k$, such that if  $q > C_k$ then every element in $M_n(\F_q)$ is a sum of two $k$th powers.
\end{quote}
Thus, $G'(k,n) = 2$ for all $k \geq 2$ with the threshold for the field size depending only on $k$. Recently, Braez \cite{Br} improved on this result and proved that, for $q \neq 2$, every matrix in $M_n(\F_q)$ is a sum of two $k$th powers whenever $q^n > (k-1)^4$. For $q = 2$ and odd $k$, the same conclusion holds in every dimension $n$; see \cite[Theorem $5.1$]{KVZ}: the argument also treats even exponents over $\F_2$ in sufficiently large dimensions and gives a uniform bound on the cardinality of the matrix ring.

However, it is evident that  above quoted result is not analogous to the group-theoretic result \cite{LST} mentioned above. Indeed, the result on finite simple groups establishes a bound on the \tit{order of the group}, but in the additive case, we established analogous result with the constraint on the field size $q$ and not on the \tit{order of matrix ring} $M_n(\F_q)$.  This article improves on all the results discussed above and proves the following (stronger) conjecture of Larsen: 

\begin{quote}
Let $n , k$ be positive integers and $q$ be a prime power. There exists a real constant $C_k$ depending only on $k$, such that if  $\# M_n(\F_q) > C_k$, then every element in $M_n(\F_q)$ is a sum of two $k$th powers.
\end{quote} 

\noindent More precisely,
\begin{thm}\label{mainthm}
For $k \geq 1$, set 
$$
D_k = (4k)^8, \qquad L_k = \lceil\log_2 D_k \rceil, \qquad C_k = D_k^{L_k}.
$$
If $\#M_n(\F_q) > C_k$ then every matrix in $M_n(\F_q)$ is a sum of two $k$th powers.
\end{thm}

\noindent We prove the result in two steps.
\begin{enumerate}
  \item \label{step1}
  Let $n \geq 2$, and suppose that 
  $$
  q > 2, \quad \trm{ or} \quad n =2,\quad \trm{ or} \quad n \geq 5.
  $$ Let $X$ be a conjugacy class whose characteristic polynomial is  irreducible  of degree $n$, and let $Y$ be a conjugacy class whose  characteristic polynomial is of the form $(t-\lam) h(t)$, where $\lam \in \F_q$, $h$ is irreducible of degree $n-1$ and $h(\lam) \neq 0$. Then, every \tit{nonscalar} $z$ satisfying $\tr(z) = \tr(X) + \tr(Y)$ belongs to $X + Y$.
  \item \label{step2}
  For $n \geq 3$, and $q^n > D_k$, we then construct classes $X$ and $Y$ as in $(1)$ consisting of $k$th powers. Scalar matrices and dimensions $1$ and $2$ are treated separately.
\end{enumerate}

Some comments regarding the existence of the above result in the literature are in order. At the end of writing this article, we learned that, very recently, Braez \cite{Br} proved an \tit{almost} similar result. So, let us observe some similarities and differences in our approach with that of Braez, and also how some of our results are relatively \tit{stronger}. 

Braez \cite{Br} combines Frobenius canonical form with completion results for unreduced upper Hessenberg matrices to construct, over $\F_q$ with $q > 2$ (and so $q \neq 2$), additive decomposition whose summands $x$ and $y$ have prescribed characteristic polynomials $f(t)$ and $(t- \lambda) h(t)$ as above and satisfying  the trace condition $\tr(z) = \tr(x) + \tr(y)$. Then he uses finite-field estimates to choose irreducible $f$ and $h$ so that both summands $x$ and $y$ are $k$th powers. Scalar matrices are handled separately through the embedding $\F_{q^n} \hookrightarrow M_n(\F_q)$.
 
In this article, we present two approaches to additive decompositions of nonscalar matrices into elements of prescribed conjugacy classes, subject to the necessary trace condition. Both use the cyclic-vector characterization established in \cite{KM}, which differs from Braez's approach through Frobenius canonical form and unreduced Hessenberg matrices. We now describe the two approaches to prove step $(1)$ above, and it will become evident that our approach is distinct from that of Braez.

Let $X$ and $Y$ be conjugacy class as in step $(1)$ above.  Every $y \in Y$ has a unique $\lambda$-eigenline. For nonscalar $z$ with matching trace, define 
$$
\pi_z : \Set{(x,y) \in X \times Y : x + y = z} \longrightarrow \P(V), \qquad (x,y) \mapsto \ker(y - \lam I)
$$
where the target is the space of one-dimensional subspaces of $V = \F_q^n$. Thus $\pi_z^{-1}(L)$ consists of decompositions in which the second summand has $\lambda$-eigenline $L$. As it turns out, fibers over $L$ are empty precisely when $L$ is $z$-invariant; see the beginning of \S \ref{adddecomp}.  For $q > 2$, the first proof provides an existence theorem and also offers a \tit{quantitative result}. By counting operators for which a specified vector fails to be cyclic, we obtain an explicit positive lower bound: for every non $z$-stabe line $L$ (which are \tit{plenty} in number for a given \tit{nonscalar} $z$), 
$$
\pi_z^{-1}(L) \geq \frac{1}{9} q^{(n-1)(n-2)}.
$$
See the end of proof of Proposition \ref{x+yprop}. Summing these bounds gives a lower bound on the total number of decompositions. But the first proof does not establish positivity in the case $q = 2$, so we need to approach differently in order deal with this exceptional case. Given a fixed non-$z$-invariant line $L$, choose a complement $W$ of $L$ in $V$, and put $m : = n-1$, and write $z = \begin{pmatrix} a & b \\ c & A \end{pmatrix}$ relative to the decomposition $V = L \oplus W$. Since $L$ is not $z$-invariant, $c \neq 0$. We use our results in \cite{KM}, wherein the characterization reduces the problem to finding an operator $B$ on $W$ with irreducible characteristic polynomial $h$ such that $c$ is cyclic for $A-B$. Once such a $B$ is found, the remaining row of $y$ is uniquely determined and furthermore $x = z-y$ belongs to $X$; see \cite{KM}. 

To construct $B \in \End_{\F_q}(W)$, we choose a basis $(v_0 = c, v_1, \ldots, v_{m-1})$ recursively so that $S(v_i) := Av_i - v_{i + 1}$, for all $0 \leq i \leq m-2$ defines a simple map on the hyperplane $H = \textrm{span}(v_0, \ldots, v_{m-2})$. Here simplicity means that no nonzero subspace of $H$ is invariant under $S$. An extension theorem of Ram \cite{Ra} for simple maps then gives a unique extension $B$ with characteristic polynomial $h$. Consequently, $((A-B)^j c = v_j)$ for $0 \leq j \leq m-1$, proving the required cyclicity, as desired. 

For $q > 2$, the recursive construction above follows by excluding choices that either don't yield simplicity or make the chosen vectors (at the inductive step) dependent. Over $\F_2$, this counting argument may leave no available choice at the final step. So we have to proceed differently in this exceptional case. In this case, we retain a suitable vector outside the span at each step, ensuring that the partial map already sends some vector outside the required hyperplane $H$. A geometric extension lemma shows that the images giving simple extensions cannot all lie in a single affine translate of $H$. This allows the final extension to preserve simplicity while placing the last basis vector outside $H$. Together with the immediate case $m =1$, this proves the decomposition theorem for $q > 2$, and for $q = 2$ when $n =2$ or $n \geq 5$.

Finally, combined with the usual finite-field trace estimates, this establishes the matrix Waring theorem for all finite fields, including \tit{even exponents} over $\F_2$ in sufficiently large dimensions; compare with our \cite[Theorem $5.1$]{KVZ} where the same is proved for odd $k$ in case $q = 2$. The paper therefore contributes quantitative decomposition bounds and a construction that is distinct from Braez's construction; furthermore this article resolves the exceptional $\F_2$ case also required to complete Larsen's conjecture.

The paper is organized as follows. In \S \ref{adddecomp}, we use our results in \cite{KM} and prove step $(1)$ above. In \S \ref{addchar}, we use standard results from character theory and prove step $(2)$ above. In \S \ref{mainthmsec}, we prove Theorem \ref{mainthm}.

\section{Additive decompositions in $M_n(\F_q)$}\label{adddecomp}

In this section, we use our results in \cite{KM}. First, let us recall the necessary setup. Let $q$ be a prime power, and $F := \F_q$ be the finite field with $q$ elements. Let $n \geq 2$ and $m := n-1$. Let $V := F^n$ be the $n$-dimensional vector space over $F$, and $\P(V)$ be the associated projective space. Let $X$ be a conjugacy class in $M_n(F)$ whose characteristic polynomial is irreducible of degree $n$ and let $Y$ be a conjugacy class in $M_n(F)$ whose characteristic polynomial has the form $(t- \lambda) h(t)$, where $\lambda \in F$ and $h$ is irreducible of degree $m$ and $h(\lambda) \neq 0$. 


The conjugacy class $Y$ can be partitioned into subsets $S_{L}$ parametrized by $L \in \P(V)$, where $S_{L}$ consists of elements $y \in Y$ which have the (unique) $\lambda$-eigenline $L$. For $y \in Y$, consider one such associated set $S_{L_y}$, $L_y \in \P(V)$. As it turns out \cite[Lemma 4.1]{KM}, after fixing a decomposition $V  = L_y \oplus W$, the set $S_{L_y}$ can be parametrized by $\CalC_h \times W^*$, where 
$$
\CalC_h = \Set{B \in \End_{F}(W) : \trm{ characteristic polynomial of B is } h(T)}.
$$
Since $h$ is irreducible of degree $m$, it follows that \cite[Lemma $4.1$]{KM} $\CalC_h$ is a single $GL(W)$-conjugacy class in $\End_{F}(W)$. Furthermore, every element of $S_{L_y}$ is of the form $\begin{pmatrix} \lam & r  \\ 0 & B \end{pmatrix}$ where $r \in W^*$ and $B \in \CalC_h$. 

Suppose $z$ is a nonscalar with trace $\tr(z) = \tr(X) + \tr(Y)$ and let $z := \begin{pmatrix} a & b \\ c & A \end{pmatrix}$ with respect to the decomposition $V = L_y \oplus W$, where $b \in W^*, c \in W, A \in \End_F(W)$.  Now, consider the map 
$$
\pi_z : \Set{ (x,y) \in X \times Y : x + y = z} \to \P(V), \qquad
(x,y) \mapsto L_y
$$
where $L_y$ is the unique $\lambda$-eigenline of $y$.  If $L_y$ is $z$-stable, then $\pi_z^{-1}(L_y) = \emptyset$. But, there are \tit{many} lines in $\P(V)$ which are not $z$-stable \cite[proof of Theorem $1.2$]{KM}; indeed, if $G_z$ denotes the set of lines which are not $z$-stable then $\abs{G_z - q^{n-1}} \leq 6 q^{n-2}$.  

The aim of \cite{KM} was different from that of this paper. In \cite{KM}, we were interested in estimating the $number$ of decompositions of $z$ as the sum $x + y$ with $x \in X$ and $y \in Y$ and prove that the number is approximately independent of the nonscalar matrix $z$, subject to the trace condition. In this article, we are interested to prove the $existence$ of a decomposition of $z$ as $x + y$, with $x \in X$ and $y \in Y$. In other words, in this paper we would like to show that  $\# \pi_z^{-1}(L_y) > 0$ for \tit{some} $y \in Y$. We give two proofs for this assertion. The first proof actually provides a lower bound on the \tit{number} of decompositions $z = x + y$ but it does not apply when $q = 2$, while the second one proves only the \tit{existence} of the decomposition but applies for every $q$.

\subsection{Quantitative analysis of additive decompostions}
Instead of counting the matrices $B \in \CalC_h$ for which $c$ is a cyclic vector for $A- B$, we count those elements $B \in \CalC_h$ for which $c$ is \tit{not} a cyclic vector in $A- B$, and subtract it from the cardinality of $\CalC_h$. Since our goal in \cite{KM} was different from here, there we obtained \tit{non-optimal} bounds. Here we estimate $\#\pi_z^{-1}(L)$ more carefully in order to show that $\pi_z^{-1}(L) > 0$ for all $q > 2$, thus establishing the existence of a decomposition $z = x + y$, with $x \in X$ and $y \in Y$. Thus it suffices to establish a tighter upper bound for $Z_{A,c}$, for $A \in \End_F(W)$ and $0 \neq c \in W$, which is defined by
$$
Z_{A,c} := \Set{B \in \CalC_h : c \trm{ is not cyclic for } A - B}.
$$
We begin with a preliminary lemma.  Let $K \sub U \sub V$ be subspaces of $V$. Recall that a map $T: K \to U$ is simple if there is no \tit{nonzero} subspace $S \sub K$ such that $T(S) \sub S$.

\begin{lem}\label{simplemapsnumber}
Let $U$ be a subspace of $V$ of dimension $r$, and $K$ be a proper subspace of $U$, so that $\dim K = r  - s$ for some $s\geq 1$.  Then
$$
\# \Set{ T: K \to U : T \trm{ is simple }} = \prod_{j = 1}^{r-s} (q^r - q^j).
$$
\end{lem}
\begin{proof}
Fix one monic \tit{irreducible} polynomial $g(T)$ of degree $r$. Let $$\CalC_g = \Set{ B \in \End_{F}(U) : \chi_B(T) = g(T)}.$$ These operators form a single conjugacy class, and its size is 
$$
\# \CalC_g = \frac{\# GL_r(q)}{q^r - 1} = \prod_{j = 1}^{r-1}(q^r - q^j),
$$
since the centralizer in $GL_r(q)$ is the isomorphic to the multiplicative group of the field $F[B] \cong \F_{q^r}$, and hence has order $q^r -1$.

Every restriction $B|_K$, with $B \in \CalC_g$, is simple. Indeed, suppose a nonzero $S \sub K$ satisfies $(B|_K)(S) \sub S$. Then $S$ is also a nonzero proper $B$-invariant subspace of $U$, because $S \sub K \subsetneq U$. In a basis beginning with a basis of $S$, the matrix of $B$ is block upper triangular. Consequently, $\chi_B = g$ factors into two polynomials of positive degree, contradicting irreducibility of $g$. 

We therefore have a restriction map 
$$
\rho: \CalC_g \to \Set{ \trm{simple maps } K \to U}, \qquad B \mapsto B|_K.
$$
By  \cite[Lemma $2.7$]{Ra}, given a simple map $T : K \to U$, the number of extensions $B$ with $\chi_B = g$ is given by 
\begin{equation}\label{simplemaps}
\# \rho^{-1}(T) = \prod_{j = r-s+1}^{r-1}(q^r - q^j).
\end{equation}
This number is positive (with the empty product taken to be $1$) and is the same for every simple $T$. Thus, every simple maps occurs as a restriction, each simple map occurs exactly the number of times given by \eqref{simplemaps}, and different simple maps have disjoint sets of extensions. Thus, if $M$ denotes the number of distinct simple maps $T: K \to U$, then 
$$
\prod_{j = 1}^{r-1} (q^r - q^j) = M \prod_{j = r - s + 1}^{r-1}(q^r - q^j),
$$
so that 
$$	
M = \frac{\prod_{j = 1}^{r-1}(q^r - q^j)}{\prod_{j = r-s+1}^{r-1}(q^r - q^j)} = \prod_{j =1}^{r-s}(q^r - q^j)
$$
as desired.
\end{proof}
\noindent For $j \geq 1$ let 
$$
P_j(q) := \prod_{i = 1}^j \left( 1 - \frac{1}{q^i} \right), \qquad P_0(q) :=1.
$$
\begin{lem}\label{fau}
Let $W$ be a subspace of $V$ of dimension $m$, and let $A \in \End_F(W)$ and $U$ a nonzero proper subspace of $W$ of dimension $r$, so that $1 \leq r < m$. Let $d := m-r$, and $\tau := \pi \circ A |_U : U \to W/U$, and $s$ be the rank of $\tau$. Define
$$
\CalF_{A,U} = \Set{ B \in \CalC_h : (A-B)U \sub U}.
$$
Then 
$$
\# \CalF_{A,U} = 
\begin{cases} 
0, & \trm{ if } s= 0, \\
q^{m^2-m-rd} \frac{P_{r-1}(q) P_{d-1}(q)}{P_{s-1}(q)}, &  \trm{ if } s \geq 1.
\end{cases}
$$
\end{lem}
\begin{proof}
Choose a complement of $U$ in $W$; it is isomorphic to $W/U$. For a given $D \in \End(U)$, a lift $T_D : U \to W$ of $\tau$ (depending on the choice of $D$) can then be written as $T_D(u) = Du + \tau u $  (after identifying $W/U$ with the complement of $U$ in $W$). 
Consider the following diagram, where $\pi$ is the usual projection map, $K = \ker(\tau)$ and $\iota$ is the usual inclusion.
\begin{center}
\begin{tikzcd}[row sep=large, column sep=large]
  & W \arrow[r, dotted] & W \arrow[d, "\pi"] \\
K \arrow[r, hook, "\iota"]
  & U \arrow[u, hook] \arrow[ur, dotted, "T_D"] \arrow[r, "\tau"']
  & W/U
\end{tikzcd}
\end{center}

We have $\dim K = r-s$. The map $T_D$ is simple if and only if $D|_K : K \to U$ is simple. Indeed, suppose $U'$ be a nonzero $T_D$-invariant subspace of $U$. Then $\pi \circ T_D (U') =  \tau(U') = 0$, so that $U'$ is contained in $K$. On the other hand $T_D$ agrees with $D$ on $K$ so that $D|_{K}(U') \sub U'$ and $D|_{K}$ is not simple. The converse is clear.

Thus we count allowable $D$'s (so that $T_D$ is simple) in two stages. First, note that, for $s \geq 1$, it follows from Lemma \ref{simplemapsnumber} that the number of simple maps $S: K \to U$ is $\prod_{j =1}^{r-s}(q^r - q^j)$. Now, secondly, for each such choice of $S$,  we want to count linear maps $D: U \to U$ satisfying $D|_K = S$. Choose a basis $e_1, \ldots, e_{r-s}$ of $K$, and extend it to a basis $e_1, \ldots, e_{r-s} e_{r-s+1}, \ldots, e_r$ of $U$. The condition $D|_K = S$ already determines $D(e_i) = S(e_i)$, $1 \leq i \leq r-s$. But the images of the remaining $s$ basis vectors can be chosen arbitrarily in $U$. Each image has $q^r$ possible values since $\abs{U} = q^r$. Every such choice determines a unique linear extension, so the number of extensions is $(q^r)^s = q^{rs}$.

It follows that the number of simple lifts of $\tau$ is exactly the product of number of choices of $S$ and for each choice of $S$ the number of extension $D$: 
\begin{align*}
\# \Set{ \trm{simple lifts of } \tau } &=  \prod_{j =1}^{r-s}(q^r - q^j) \cdot q^{rs}  = q^{rs} q^{r(r-s)} \prod_{j =1}^{r-s} (1-q^{j-r})\\
&= q^{r^2} \prod_{j =1}^{r-s} (1- q^{-(r-j)}) = q^{r^2}\frac{P_{r-1}(q)}{P_{s-1}(q)}.
\end{align*}

The condition $(A-B)U \subset U$ in the definition of $\CalF_{A,U}$ is equivalent to $\pi \circ B|_U = \pi \circ A|_U$. Indeed, if for all $u \in U$, $(A-B)u \in U$ it follows that $Au + U = Bu + U$ in $W/U$ so that $\pi \circ B|_U = \pi \circ A|_U$, and conversely. Therefore we want to count those $B \in \CalC_h$ such that $\pi \circ B|_U  = \tau$, in other words, the extensions in $\CalC_h$ of the simple lifts of $\tau$.
Each simple lift $T_D$ of $\tau$ has $\prod_{j = r+1}^{m-1}(q^m -q^j) = q^{m(d-1)} P_{d-1}(q)$ extensions in $\CalC_h$, again by Lemma $2.7$ \cite{Ra}; see the diagram above. Thus 
$$
\# \CalF_{A,U} = q^{m(d-1)} P_{d-1}(q) \cdot q^{r^2}\frac{P_{r-1}(q)}{P_{s-1}(q)},
$$
proving the equality.  Finally, if $s = 0$, every lift preserves $U$, so none can extend to an irreducible operator in $\CalC_h$.
\end{proof}

\begin{prop}\label{x+yprop}
For every $q > 2$, and for every nonscalar matrix $z$ such that $\tr(z) = \tr(X) + \tr(Y)$, there exists $x \in X$ and $y \in Y$ such that $z = x + y$. 
\end{prop}
\begin{proof}
Because $s \leq d$, the formula in Lemma \ref{fau} gives the convenient bound
$$
\abs{\CalF_{A,U}} \leq q^{m^2-m-r(m-r)}P_{r-1}(q).
$$
Since
$$	
Z_{A,c} = \bigcup_{F c \subseteq U  \subsetneq W} 
\CalF_{A, U},
$$
we have
$$
\abs{Z_{A,c}} \leq \sum_{r = 1}^{m-1}\begin{bmatrix} m-1 \\ r-1 \end{bmatrix}_q	 q^{m^2-m-r(m-r)} P_{r-1}(q).
$$
For $d = m-r$, the exact Gaussian coefficient is 
$
\begin{bmatrix}m-1 \\ r-1 \end{bmatrix}_q = q^{(r-1)d} \frac{P_{m-1}(q)}{P_{r-1}(q) P_d(q)}.
$
Thus
$$
\abs{Z_{A,c}} \leq q^{m^2-m}P_{m-1}(q) \sum_{d = 1}^{m-1} \frac{q^{-d}}{P_d(q)}.
$$
Since $\frac{q^{-d}}{P_d(q)} = \frac{1}{P_d(q)} - \frac{1}{P_{d-1}(q)}$, the sum telescopes, and we get 
$$
\abs{Z_{A,c}} \leq q^{m^2-m}(1- P_{m-1}(q)).
$$
Consequently every line that is not $z$-stable, we have 
\begin{align*}
\# \pi_z^{-1}(L) &= \abs{\CalC_h} - \abs{Z_{A,c}}\\
& \geq q^{m^2-m}P_{m-1}(q) - q^{m^2 -m}(1- P_{m-1}(q))\\
& = q^{m^2-m}(2 P_{m-1}(q) - 1).
\end{align*}
For $q \geq 3$, 
$$
P_{m-1}(q) \geq 1 - \frac{1}{q} - \frac{1}{q^2} \geq \frac{5}{9}.
$$ 
Indeed, separating the first factor and using $\prod (1- a_i) \geq 1 - \sum a_i$ on the remaining factors. 
Hence $\#\pi_z^{-1}(L) \geq \frac{1}{9}q^{m^2 - m} > 0$ as desired.

\end{proof}

\begin{remark}
Note that for $q = 2$, the quantity $2P_{m-1}(2) -1$ equals zero when $m =2$ and is strictly negative for every $m \geq 3$. Consequently, the estimate in the proof above does not establish positivity of the fibers when $q = 2$ and $m \geq 2$. The case $m = 1$, corresponding to $n =2$, however goes through. 
\end{remark}

\subsection{Existence of additive decompositions}

Given a matrix $B \in \CalC_h$ we can form the matrix $y_{B,r} := \begin{pmatrix} \lambda & r \\ 0 & B \end{pmatrix}$ where $\lambda \in F$,  which belongs to $Y$, and furthermore, we have $z-y_{B,r} \in X$ for some  $r \in W^*$ precisely when $c$ is a cyclic vector for $A - B$ \cite[Proposition $4.6$]{KM}. Most importantly, in this case, $r$ is unique. Thus we observe that among the lines which are not $z$-stable, the cardinality of the fiber is given by
$$
\#\pi_z^{-1}(L) = \# \Set{B \in \CalC_h : c \trm{ is cyclic for } A - B}.
$$

The section is devoted to proving that there exists $B \in \CalC_h$ such that $c$ is a cyclic vector for $A-B$, so that $\# \pi_z^{-1}(L) > 0$, which implies the existence of a decomposition $z = x + y$, where $x \in X$ and $y \in Y$ as desired.
Thus, we have reduced our problem to demonstrating the existence of a $B \in \CalC_h$ such that $c$ is cyclic for $A-B$. We carry out this construction inductively beginning with $c$ and ending with a $B \in \CalC_h$ such that $c$ is cyclic for $A-B$. 

For a subspace $U \subsetneq W$, a linear map $S : U \to W$ is called \tit{simple} if there is no nonzero subspace $L \subseteq U$ satisfying $S(L) \subseteq L$. We use two extension facts:

\begin{lem}(\cite[Lemma $2.10$]{ARV})\label{ram1}
 Let $R \subsetneq U \subsetneq W$ be subspaces of $W$, where $\dim R = r$,  $\dim U = r +1$, $\dim W = m$ and $0 \leq r \leq m-2$. Then every simple map $S : R \to W$ has exactly $q^m - q^{r+1}$ simple extensions to $U$. Thus, after fixing $v \in U \setminus R$, exactly $q^{r+1}$ possible values of $S(v)$ (in $W$) produce nonsimple extensions.
\end{lem}

\begin{lem}(\cite[Lemma $2.7$]{Ra})\label{ram2}
  Every simple map on a hyperplane $H \sub W$ has a unique extension in $\CalC_h$. 
\end{lem}

First we prove a preliminary result which is required to deal with the exceptional case $q = 2$; the reason why $q = 2$ is exceptional will become evident in the proof of Proposition \ref{mainprop} below . Suppose we have a simple map $S : R \to W$, and let $H$ be a hyperplane in $W$ containing $R$. The goal is to extend $S$ to $H$. If $H = R \oplus \F_2 v$, for some $v \in H$, then we only need to decide $S_w(v) = w \in W$. Such a choice $w$ is ``good'' if the resulting extension remains simple, otherwise it is ``bad''. The following lemma tells us that if $S$ already sends some vector of $R$ outside $H$, then the good choices cannot all lie in any one translate of $H$. 

\begin{lem}\label{binaryfinalextension}
Suppose $q = 2$, and subspaces $R \sub H \sub W$ such that $\dim R = m -2$, $\dim H = m-1$, and $\dim W = m$. Let $S : R \to W$ be simple, and suppose $S(R) \not \subseteq H$. Write $H = R \oplus \F_2 v$. For any $a \in W$, there is a simple extension $\widetilde{S} : H \to W$ such that $\widetilde{S} (v) \not \in a + H$.
\end{lem}
\begin{proof}
For $x \in W$, let $S_x$ extend $S$ by $S_x(v) = x$, and put 
$$
E = \Set{ x \in W : S_x \trm{ is simple}}, \qquad D = W \setminus E.
$$
By Lemma \ref{ram1}, $\abs{E} = 2^m -2^{m-1} = 2^{m-1}$. We show that $E$ cannot be an entire affine hyperplane. Suppose not. Over $\F_2$, its complement $D$ is also an affine hyperplane. Since $S_0(v) = 0$ by definition, $S_0$ is not simple, hence $0 \in D$ and so $D$ is a \tit{linear} hyperplane. On the other hand, note that $v , S(r), S(r) + r + v \in D$ for all $r \in R$. Indeed, $S_v(v) = v$ so that $S_v$ is not simple, hence $v \not \in E$. Similarly, since $v + r \neq 0$ because $v \not \in R$, we have  $S_{S(r)}(v + r) = S(r) + S(r) = 0$ (we are in characteristic $2$), and $S_{S(r) + r + v} (v + r) = v+r$, so that $S_{S(r)}$ and $S_{S(r)+r + v}$	 cannot be simple. Now, since $D$ is linear, it also contains $r = S(r) + (S(r) + r + v) + v$, so that $R \sub D$, and since $v \in D$ it follows that $H \sub D$. Since both $H$ and $D$ are hyperplanes, $H = D$. But $S(r) \in D$ for every $r \in R$, so that $S(R) \subseteq H$, a contradiction. 

Therefore $E$ is not an affine hyperplane. If $E \subseteq a + H$, then since $\abs{E} = 2^{m-1}$ and $\abs{a + H} = 2^{m-1}$ it follows that $E = a + H$ and $E$ is a hyperplane, again a contradiction. Thus, for every $a \in W$, some ``good'' image lies outside every $a + H$.
\end{proof}

We now establish the ``essentially'' main result of this article. 
\begin{prop}\label{mainprop}
Let $m \geq 1$, $W$ be an $m$-dimensional vector space over $\F_q$,  $A \in \End_{\F_q}(W)$, and $0 \neq c \in W$. Suppose 
$$
m = 1,\qquad  \trm{ or } \quad q > 2, \qquad \trm{ or } \quad m \geq 4.
$$
For every monic irreducible $h$ of degree $m$, there exists $B \in \CalC_h$ such that $c$ is cyclic for $A - B$.
\end{prop}
\begin{proof}
If $m =1$, every nonzero vector is cyclic for every operator, in particular for $A - B$. 

For $m \geq 2$, we seek a basis $(v_0 = c, v_1, \ldots, v_{m-1})$ of $W$ such that the maps
$$
S_i : R_{i+1} \to W, \qquad S_i(v_j) = Av_j - v_{j+1} \quad (0 \leq j \leq i),
$$
are simple for $0 \leq i \leq m-2$, where $R_0 = 0$ and $R_j = \trm{span}(v_0, \ldots, v_{j-1})$. Each $S_i$ extends $S_{i-1}$, with $S_{-1} :\Set{0} \to W$. In particular, $S_{m-2}$ is simple on the hyperplane $H = R_{m-1}$. By Lemma \ref{ram2}, it extends uniquely to $B \in \CalC_h$. Then $(A-B)v_i = v_{i+1}$, for $0 \leq i  \leq m-2$, so that $c$ is cyclic for $A - B$.

For $q > 2$, begin with $v_0 = c$ and the simple zero map $S_{-1} : \Set{0} \to W$ on the zero subspace.  First, we choose $v_1$ (to be decided later) and then assign the image of $v_0$. For any candidate $v_1 \in W$, define $S_0(v_0) = A v_0 - v_1$. This determines $S_0$ on the subspace generated by $v_0$. Now we choose $v_1$ so that $v_1 \not \in \trm{span}(v_0)$ and such that $S_0$ is simple. These conditions can be checked for each candidate $v_1 \in W$. In particular, at this step, simplicity means $Av_0 - v_1 \not \in \trm{span}(v_0)$. At the end of this step, $v_1$ is known, but $S_0(v_1)$ \tit{has not been defined}, since the domain of $S_0$ is still the subspace generated by $v_0 = c$. Now, by induction on $r$, we want to extend $S_{r-1}$ on $R_{r} := \trm{span}(v_0, \ldots, v_{r-1})$ to a map $S_{r}$ on $R_{r+1} = \trm{span}(v_0, \ldots, v_r)$, so suppose, for $r \geq 1$, $v_0, \ldots, v_{r-1},v_r$ have been chosen so that the ordered set $(v_0, \ldots, v_{r-1},v_r)$ is linearly independent and $S_{r-1}$ is simple on $R_r := \trm{span}(v_0, \ldots, v_{r-1})$. Choosing $v_{r+1}$ prescribes $S_r(v_r) = A v_r - v_{r+1}$, hence defines an extension of $S_{r-1}$ to $S_r$ by defining on $v_r$. But, by Lemma \ref{ram1}, exactly $q^{r+1}$ choices in $W$ (of $v_{r+1}$) produce a \tit{nonsimple} extension, and furthermore also $q^{r+1}$ choices lie in $R_{r+1}$ making the new vectors dependent; these sets may overlap. Hence at least 
$$
q^m -2 q^{r+1} \geq q^{m-1}(q-2) > 0
$$
choices work for every $0 \leq r \leq m-2$. Thus, finite induction constructs the desired basis and a simple map $S_{m-2}$ on $H$. By Lemma \ref{ram2}, this simple map extends to a unique $B \in \CalC_h$ satisfying the desired properties. 

Now suppose $q = 2$ and $m \geq 4$. The preceding estimate can become zero, so we take care by keeping a reserved vector $\omega$ outside the chosen span at each induction step. We first describe the induction step, and then the initial choices that must be made to start the induction. Suppose $v_0, \ldots, v_{r-1},v_r$ are independent,
$$
\omega \not \in R_{r+1} = \trm{span}(v_0, \ldots, v_{r-1},v_r),
$$
and $S_{r-1} : R_r \to W$ is simple. In choosing $v_{r+1}$, we want to avoid the nonsimple choices and also the set $R_{r+1} + \F_2 \omega$, having sizes $2^{r+1}$ and $2^{r+2}$ respectively. Indeed, the set $R_{r+1}+ \F_2 \omega$ is the disjoint union
$$
R_{r+1} + \F_2 \omega = R_{r+1} \sqcup (\omega + R_{r+1})
$$
If we choose $v_{r+1} \in R_{r+1}$, the enlarged set of vectors is linearly dependent. Moreover, if $v_{r+1} \in \omega + R_{r+1}$, so that $v_{r+1} = \omega  + u$ for some $u \in R_{r+1}$, then $\omega = u + v_{r+1}$. Hence $\omega$ belongs to $\trm{span}(v_0, \ldots, v_r, v_{r+1})$. Thus avoiding $R_{r+1} + \F_2 \omega$ for the choice of $v_{r+1}$ keeps the chosen vectors independent and keeps $\omega$ outside their span. Put together,  for $r \leq m-3$, since
$$
2^{r+1} + 2^{r+2} \leq 3 \cdot 2^{m-2} < 2^m	
$$ 
an allowable choice always exists. We make the initial choices needed to begin the construction. We split into three cases. 
\begin{enumerate}
  \item 
  Suppose $Ac \not \in \F_2 c$. Put $\omega = Ac$, and choose $v_1 \not \in \trm{span}(c, \omega)$. For $R_1 = \F_2 c$, defining $S_0(c) = \omega + v_1$ gives a simple map $S_0 : R_1 \to W$, and $\omega \not \in \trm{span}(c,v_1)$. Applying the construction above for $r = 1, \ldots, m-3$, we get at the end 
  $$
  R_{m-2} = \trm{span}(v_0, \ldots, v_{m-3}), \qquad H = R_{m-1} = \trm{span}(v_0, \ldots, v_{m-2})
  $$
  with $S_{m-3} : R_{m-2} \to W$ simple and $\omega \not \in H$. Since $v_1 \in H$, we have $S_{m-3}(c) = \omega + v_1 \not \in H$, and hence $S_{m-3}(R_{m-2}) \not \subseteq H$. Applying Lemma \ref{binaryfinalextension} with $v = v_{m-2}$, and $a = A v_{m-2}$, we obtain a simple extension $S_{m-2} : H \to W$ satisfying $S_{m-2}(v_{m-2}) \not \in A v_{m-2} + H$. Set
  $$
  v_{m-1} := Av_{m-2} + S_{m-2}(v_{m-2}).
  $$
  Then $v_{m-1} \not \in H$, and $(v_0, \ldots, v_{m-1})$ is a basis for $W$ and $S_{m-2}$ satisfies $S_{m-2}(v_i) = Av_i - v_{i+1}$ for $0 \leq i \leq m-2$. Now, applying Lemma \ref{ram2} we obtain a unique extension of $S_{m-2}$ to a $B \in \CalC_h$.
  \item
  Suppose $Ac \in \F_2c$, and the induced operator on  $W/\F_2c$ is nonscalar. This case starts with a two-dimensional map. Let $Ac = \alpha c$, where $\alpha \in \F_2$. Choose $v_1$ such that $Av_1 \not \in \trm{span}(c,v_1)$, and put $\omega = A v_1$; such a vector exists because the induced operator on $W/\F_2 c$ is nonscalar. Since $m \geq 4$ (here we use the hypothesis that $m \geq 4$), choose $v_2 \not \in \trm{span}(c,v_1, \omega)$. Define $S_1 : R_2 \to W$ by
  $$
  S_1(c) = \alpha c + v_1, \qquad S_1(v_1) = \omega + v_2.
  $$
  This map on $U = \trm{span}(c,v_1)$ is simple. Indeed, the kernel of $U \xrightarrow{S_1} W \to W/U$ is precisely $\F_2c$. Every invariant subspace in $U$ must therefore lie in $\F_2 c$, while that line is not invariant because $S_1(c) \not \in \F_2 c$.
  
  Also $\omega \not \in \trm{span}(c,v_1, v_2)$.  If $m > 4$, apply the above induction for $r = 2, \ldots, m-3$; if $m =4$, no induction step is needed. In both cases, put 
  $$
  R_{m-2} = \trm{span}(v_0, \ldots, v_{m-3}), \qquad H = \trm{span}(v_0, \ldots, v_{m-2})
  $$
  Then $S_{m-3} : R_{m-2} \to W$ is simple, $\omega \not \in H$, and $v_1 \in R_{m-2}$. Since $v_2 \in H$, $S_{m-3}(v_1) = \omega + v_2 \not \in H$. Thus $S_{m-3}(R_{m-2}) \not \subseteq H$. Now Lemma \ref{binaryfinalextension} completes the construction as in the first case.

  \item
  Suppose $Ac \in \F_2 c$, and the induced operator on $W/\F_2 c$ is scalar. Then $A = a I + c \phi$ for some $a \in \F_2$ and $\phi \in W^*$. Indeed, since the induced operator on $W/\F_2c$ is scalar, $(A- aI)W \subseteq \F_2 c$, so that $A = aI + c \phi$ for some $\phi \in W^*$. Choose any $B \in \CalC_h$. Irreducibility of $h$ implies that every nonzero vector is cyclic for $B$, because the subspace generated by $B^i c$, with $i \geq 0$, is a nonzero  invariant subspace, which must equal $W$ since $h$ is irreducible.  Put 
  $$
  T = A - B = a I + B + c \phi.
  $$
  For every $j \geq 1$, induction gives
  $$
  T^j c - B^j c \in \trm{span}(c, Bc, \ldots, B^{j-1}c).
  $$
  Indeed, $T (B^ic) = B^{i+1}c + a B^i c + \phi(B^ic)c$. Consequently, $\trm{span}(c, Tc, \ldots, T^j c) = \trm{span}(c, Bc, \ldots, B^jc)$. Taking $j = m-1$ proves that $c$ is cyclic for $A-B$. These three cases exhaust all possibilities and complete the proof. 

\end{enumerate} 
\end{proof}

\begin{thm}\label{x+ythm}
Let $n \geq 2$. Let $X$ be a conjugacy class in $M_n(\F_q)$ whose characteristic polynomial is irreducible of degree $n$, and let $Y$ be a conjugacy class in $M_n(\F_q)$ whose characteristic polynomial is of the form $(t- \lam) h(t)$ where $h$ is irreducible of degree $n-1$ and $h(\lam) \neq 0$. Suppose $q >2$, or $n = 2$, or $n \geq 5$. Then for every nonscalar $z$ satisfying $\tr(z) = \tr(X) + \tr(Y)$ there exists $x \in X$ and $y \in Y$ such that $z = x + y$.
\end{thm}
\begin{proof}
This follows from Proposition \ref{mainprop} and the discussion before Lemma \ref{ram1}.
\end{proof}

\section{Additive Characters and Gauss Sums}\label{addchar}

First, we prove a finite-field counting lemma, before proving the main result of this section. Let $d \geq 1$,  $Q := q^d$, and $e = \gcd(k,Q-1)$. For $t \in \F_q$, let 
$$
R_t := \# \Set{u \in \F_Q : \tr_{\F_Q/\F_q}(u^e) = t}.
$$

\begin{lem}\label{mulchar}
Let $g$ be a generator of $\F_Q^{\times}$ and $\chi: \F_Q^{\times} \to \C^{\times}$ be the multiplicative character defined by 
$$
\chi(g^j) := \exp \left( \frac{2\pi j}{e} \cdot i \right)
$$
For $v \neq 0$,
$$
\sum_{j = 0}^{e-1} \chi(v)^j = \begin{cases} e & \trm{ if } v \trm{ is an } e\trm{th power} \\ 0 & \trm{otherwise}.\end{cases}
$$
Moreover, $\chi(v) = 1$ if and only if $v$ is an $e$th power. Furthermore, $\sum_{j =0}^{e-1} \chi(v)^j$ is equal to the number of solutions of the equations $x^e = v$.
\end{lem}
\begin{proof}
Fix a generator $g$, and put $\zeta := \exp( \frac{2 \pi}{e} i )$. Define $\chi(g^r) = \zeta^r$. This is well-defined: if $g^r = g^s$, then $Q-1 \mid r-s$, and hence $e \mid r -s$, so $\zeta^r = \zeta^s$. Since $\chi(g) = \zeta$, the character $\chi$ has order $e$.

Let $v := g^r$. Then $v$ is an $e$-th power if and only if $e \mid r$. Indeed, if $v = u^e$ for some $u \in \F_Q^{\times}$, then $u = g^s$ for some $s \in \Z$, so that $g^r = g^{es}$, hence $r \equiv es \pmod{Q-1}$. Since $e \mid Q-1$, it follows that $e \mid r$. The converse is clear. Because $\zeta$ has order $e$, $\chi(v) = \zeta^r = 1$ if and only if $e \mid r$. Thus $\chi(v) = 1$ if and only if $v$ is an $e$th power. 

Now consider the sum $S = 1 + \chi(v) + \chi(v)^2 + \ldots + \chi(v)^{e-1}$. If $v$ is an $e$th power, then $\chi(v) = 1$, so $S = e$. Otherwise, let $z = \chi(v)$. Then $z \neq 1$, but $z^e = 1$, so that $(z-1)S = (z + z^2 + \ldots z^e) - (1 + z + \ldots + z^{e-1}) = z^e -1 = 0$, hence $S = 0$. 

For the final assertion, consider the homomorphism $\F_Q^{\times} \to \F_Q^{\times}$, defined by $u \mapsto u^e$, which has kernel of size $e$.
\end{proof}

We express the trace condition using additive characters, and then express the $e$th power condition under multiplicative characters. The following result is well-known, but due to lack of proper reference we include a proof for completeness.
\begin{lem}\label{gauss}
Let $\theta : (\F_Q^{\times}, \times) \to \C^\times$ be a nontrivial multiplicative character, and $\Phi : (\F_Q, +) \to \C^{\times}$ be a nontrivial additive character. Put 
$$
G := \sum_{x \in \F_Q^{\times}} \theta(x) \Phi(x).
$$
Then $\abs{G} = \sqrt{Q}$. 
\end{lem}
\begin{proof}
Indeed, 
\begin{align*}
\abs{G}^2 &= \sum_{x, y \in \F_Q^{\times}} \theta(x) \overline{\theta(y)} \Phi(x) \overline{\Phi(y)} = \sum_{x,y \neq 0} \theta \left( \frac{x}{y} \right)\Phi(x-y) = \sum_{v \neq 0} \theta(v) \cdot \sum_{y \neq 0} \Phi((v-1)y),
\end{align*}
where, in the third equality we used the substitution $x = vy$. When $v = 1$, the inner sum is $Q-1$, and $-1$ otherwise since $\sum_{u \in \F_Q} \Phi(u) = 0$. Therefore 
$$
\abs{G}^2 = (Q- 1) - \sum_{v \neq 0,1} \theta(v).
$$
Since $\theta$ is nontrivial, $\sum_{v \neq 0} \theta(v) = 0$, while $\theta(1) = 1$, so that $\sum_{v \neq 0,1} \theta(v) = -1$. Hence $\abs{G}^2 = Q$, and the desired result follows.

\end{proof}

\begin{prop}\label{rtlemma}
We have 
$$
\abs{R_t - q^{d-1}} \leq (e-1)q^{d/2}.
$$
\end{prop}
\begin{proof}
Let $\psi : (\F_q, +) \to \C^\times$ be a nontrivial additive character, so that $\psi(x+y) = \psi(x) \psi(y)$ for all $x,y \in \F_q$	 and $\psi(x) \neq 1$ for some $ x\in \F_q$. Recall the additive-character orthogonality identity: 
$$
\frac{1}{q} \sum_{a \in \F_q} \psi(av) = \begin{cases} 1, & \trm{ if } v =0 \\ 0, & \trm{ if } v \neq 0. \end{cases}
$$
Applying this identity with $v := \tr(u^e) -t$, we obtain 
\begin{align*}
R_t &= \sum_{u \in \F_Q} \frac{1}{q} \sum_{a \in \F_q} \psi \left( a (\tr(u^e) -t )\right) \\
&= \frac{1}{q} \sum_{a \in \F_q} \psi(-at) \cdot \sum_{u \in \F_Q} \psi(a \tr(u^e)).
\end{align*}
Define a new additive character on $\F_Q$, $\Psi : (\F_Q, +) \to \C^\times$ by $\Psi(v) = \psi(\tr(v))$. Since $\psi$ is nontrivial, and the trace map is surjective, it follows that $\Psi$ is nontrivial too. Moreover, since the trace map is $\F_q$-linear, we have $\psi(a \tr(u^e)) = \Psi(a u^e)$. The term corresponding to $a = 0$ contributes $Q/q = q^{d-1}$ to the sum, hence 
$$
R_t - q^{d-1 } = \frac{1}{q} \sum_{a \in \F_q^{\times}} \psi(-at) S(a), \quad \trm{ where } S(a) := \sum_{u \in \F_Q} \Psi(au^e).
$$
Now observe that, for $a \neq 0$
\begin{align*}
S(a) &= 1 + \sum_{u \in \F_Q^{\times}} \Psi( a u^e) \\
&= 1+ \sum_{v \neq 0} \sum_{\substack{u \neq 0,\\ u^e = v}} \Psi(av) \\
&= 1 + \sum_{v \neq 0} \# \Set{u \in \F_Q : u^e = v} \Psi(av) \\
&= 1 + \sum_{v \in \F_Q^{\times}} \left( \sum_{j = 0}^{e-1} \chi(v)^j \right) \Psi(av) \\
&= 1 + \sum_{v \in \F_Q^{\times}} \Psi(av) + \sum_{j = 1}^{e-1} \left( \sum_{v \in \F_Q^{\times}} \chi(v)^j \right)	 \cdot \Psi(av)
\end{align*}
where the  second equality follows by noting that if $v$ is not $e$-th power the weight is $0$, and the fourth equality follows from Lemma \ref{mulchar}.

Since $a \neq 0$ and $\Psi$ is nontrivial, $\sum_{v \in \F_Q} \Psi(av) = 0$, so that $\sum_{v \in \F_Q^{\times}} \Psi(av) = -1$. Hence $S(a) = \sum_{j =1}^{e-1} G_j(a)$, where $G_j(a) = \sum_{v \in \F_Q^{\times}} \chi(v)^j \Psi(av)$. Thus each $G_j(a)$ is a Gauss sum involving the nontrivial multiplicative character $\chi^j$ and the nontrivial additive character $v \mapsto \Psi(av)$. By Lemma \ref{gauss} and the triangle inequality, we have
$$
\abs{S(a)} \leq \sum_{j = 1}^{e-1} \abs{G_j(a)} = (e-1) \sqrt{Q}
$$
so that 
$$
\abs{R_t - q^{d-1}} \leq \frac{1}{q} \sum_{a \in \F_q^{\times}} \abs{S(a)}\leq \frac{q-1}{q} \cdot (e-1) \sqrt{Q} \leq (e-1)q^{d/2}.
$$
\end{proof}

\begin{lem}\label{mainlemma}
Let $k \geq 1$ and $d \geq 2$ and suppose that $q^{d/2 -1} > 3k$. Then, for every $t \in \F_q$, there is a $k$th power $\alpha \in \F_Q$ satisfying 
$$
[\F_q(\alpha) : \F_q] = d, \qquad \tr_{\F_Q/\F_q}(\al) = t.
$$
\end{lem}

\begin{proof}
The $e$th powers and $k$th powers in $\F_Q$ coincide, including zero. Consider
$$
R_t = \# \Set{u \in \F_Q : \tr(u^e) = t}.
$$
By Proposition \ref{rtlemma}, $\abs{R_t  - q^{d-1}} \leq (e-1)q^{d/2}$. We want to discard $u$ for which $u^e$ lies in a proper subfield of $\F_Q$. The union of these subfields has size at most 
$$
\sum_{\substack{r \mid d,\\ r < d}} q^r \leq \sum_{r = 1}^{\lfloor d/2 \rfloor} q^r \leq 2 q^{d/2}.
$$
Each element has at most $e$ preimages under the map $u \mapsto u^e$. Therefore, the number of suitable $u$ is at least 
$$
q^{d-1} -(e-1)q^{d/2} - 2 e q^{d/2} = q^{d/2}( q^{d/2-1} - (3e-1)).
$$
By hypothesis $q^{d/2-1} > 3k$, and since $3k - (3e-1) = 3(k-e) + 1 > 0$, the above expression is positive.
\end{proof}

\section{Main Theorem}\label{mainthmsec}
We now prove Theorem \ref{mainthm}, the main result  of this article. 
\begin{proof}
First suppose that $n \geq 3$ and $q^n > D_k$.  Since $n \geq 3$, 
$$
q^{\frac{n}{2} - 1 } = (q^n)^{\frac{n-2}{2n}} \geq (q^n)^{(1/6)} > (4k)^{4/3} > 3k.
$$
Let $z \in M_n(\F_q)$ be nonscalar, and put $m = n-1$. Let $u \in \F_{q^m}^{\times}$ be a primitive element and set $\beta = u^k$. Then $\beta$ has degree $m$ over $\F_q$. Indeed, if $\beta \in \F_{q^r}$ for a proper divisor $r \mid m$, then $q^{m} -1 \mid k(q^r-1)$. But $r \leq m/2$ and from $q^{n/2 -1} > 3k$ it follows that $q^{m/2}>k$ so that $k(q^r-1) \leq k(q^{m/2} - 1) < q^m-1$, a contradiction.

Let $b = \tr_{\F_{q^m}/\F_q}(\beta)$. Applying Lemma \ref{mainlemma}, with $d = n$ and $t = \tr(z)-b$, we obtain an $\al \in \F_{q^n}$ of degree $n$, which is a $k$th power with trace $t$. For a finite extension of $E/\F_q$, fix an $\F_q$-basis and let $M_\gamma$ denote the matrix of the multiplication operator $x \mapsto \gam x$ for $\gamma \in E$. Let $X$ be the conjugacy class of $M_{\alpha}$ and $Y$ be the conjugacy class of $\begin{pmatrix} 0 & 0 \\ 0 & M_{\beta} \end{pmatrix}$. These classes have exactly the required properties:
\begin{enumerate}
  \item
  $X$ has irreducible characteristic polynomial of degree $n$.
  \item
  $Y$ has characteristic polynomial $T h(T)$ where $h(T)$ is irreducible of degree $n-1$ and (certainly) $h(0) \neq 0$. 
  \item
  $\tr(X) + \tr(Y) = \tr(z)$. 
  \item
  Every element of $X$ and $Y$ is a $k$th power, since $M_{v^k} = M_{v}^k$ and conjugation preserves a matrix being a $k$th power. 
\end{enumerate}
Now, if $q>2$ we can apply Theorem \ref{x+ythm} directly to conclude that $z = x + y$, where $x \in X$ and $y \in Y$. If $q = 2$, then $2^n > D_k = (4k)^8 \geq 2^{16}$ so that $n \geq 17$. In particular $n \geq 5$ and the hypothesis of Theorem \ref{x+ythm} holds again. Thus in every case $z = x + y$ with $ x \in X$ and $y \in Y$.

Scalar matrices are handled separately. If $z = a I_n$ for $a \in \F_q$, consider the embedding 
$
\F_{q^n} \hookrightarrow M_n(\F_q)
$
by multiplication matrices $\alpha \mapsto M_{\alpha}$. By \cite[Theorem $2.8$]{KVZ}, every nonzero $a$ is a sum of two $k$th powers, when $q^n > (k-1)^4$. This proves the result when $q^n > D_k$ for $n \geq 3$. 

Now we treat the case $n = 2$. Recall that $D_k = (4k)^8$ and assume that $q^n > D_k$. When $n = 2$, this becomes $q^2 > (4k)^8$ so that $q > (4k)^4$, and since $(4k)^4 > (k-1)^4 + 6k$, it follows \cite[Theorem $1.4$]{KVZ} that for every $Z \in M_2(\F_q)$, there exist $A, B \in M_2(\F_q)$ such that $Z = A^k + B^k$. Notice that we treated $n = 2$ separately, because the condition $q^{n/2 - 1} > 3k$ does not hold for $n = 2$, so that we cannot apply Lemma \ref{mainlemma}.

For the remaining case $n = 1$, the hypothesis is $q > D_k = (4k)^8$. Clearly $D_k$ is strictly greater than $(k-1)^4$, and therefore Theorem $2.8$ \cite{KVZ} applies.

Thus, if $q^n > (4k)^8$ then $M_n(\F_q) = \Set{A^k +B^k : A , B \in M_n(\F_q)}$. The passage to the $q^{n^2} > C_k$ result is straightforward. Indeed, if $q^n \leq D_k$ then since $q \geq 2$ it follows that $n \leq \log_2 D_k \leq L_k$ and so $q^{n^2} = (q^n)^n \leq D_k^{L_k} = C_k$. Therefore if $q^{n^2} > C_k$ then $q^n > D_k$, and the result follows. 
\end{proof}

\section*{Acknowledgements}
It is a great pleasure to thank Michael Larsen for suggesting the problem and for many useful discussions.

\bibliographystyle{elsarticle-num} 

\end{document}